\documentclass{amsart}

\usepackage[table]{xcolor}

\usepackage{hyperref}

\usepackage{enumerate}
\usepackage{upgreek}
\usepackage{bm}
\usepackage{pstool}

\usepackage{thmtools}
\usepackage{thm-restate}

\usepackage{url}

	{\begin{proof}[\underline{Proof of {#1}}\textnormal{:}]}%
	{\end{proof}}

\definecolor{darkslateblue}{rgb}{0.28, 0.24, 0.55}
\definecolor{antiquefuchsia}{rgb}{0.57, 0.36, 0.51}
\definecolor{desertsand}{rgb}{0.93, 0.79, 0.69}

\DeclareFontFamily{U}{MnSymbolC}{}
\DeclareSymbolFont{MnSyC}{U}{MnSymbolC}{m}{n}
\DeclareMathSymbol{\diamonddot}{\mathbin}{MnSyC}{"7E}
\DeclareFontShape{U}{MnSymbolC}{m}{n}{
    <-6>  MnSymbolC5
   <6-7>  MnSymbolC6
   <7-8>  MnSymbolC7
   <8-9>  MnSymbolC8
   <9-10> MnSymbolC9
  <10-12> MnSymbolC10
  <12->   MnSymbolC12}{}

\newcommand{\natnumb}{\mathbb{N}^{+}}

\usepackage{amssymb}
\usepackage{amsthm}
\usepackage{tikz}

\theoremstyle{definition}

\newtheorem{theorem}{Theorem}[subsection]

\newtheorem{definition}[theorem]{Definition}
\newtheorem{lemma}[theorem]{Lemma}
\newtheorem{prop}[theorem]{Proposition}
\newtheorem{remark}[theorem]{Remark}
\newtheorem{corollary}[theorem]{Corollary}

\newtheorem{problem}{Problem}

\newcommand{\Rd}[2]{R_{#1#2}}
\newcommand{\RMd}{\Rd{\Model}{d}}
\newcommand{\RFd}{\Rd{\field}{d}}

\newcommand{\RMUd}{\Rd{\Model^\ultrafilter}{d}}

\newcommand{\RMdU}{(\Rd{\Model}{d})^\ultrafilter}

\newcommand{\CA}[1]{\mathfrak{Cs}\, #1\,}

\newcommand{\ca}[1]{\mathsf{Conc}\,#1\,}

\newcommand{\can}[1]{\mathsf{Conc}_n\, #1\, }

\newcommand{\defeq}{\ \stackrel{\textsf{\tiny def}}{=} \ }
\newcommand{\defiff}{\stackrel{\textsf{\tiny def}}{\iff}}
\newcommand{\ie}{i.e., }
\newcommand{\eg}{e.g., }

\newcommand{\Id}{\mathsf{Id}} 

\newcommand{\vv}[2]{\vec{{#1}}_{#2}}

\newcommand{\flatten}{\mathsf{flatten}}
\newcommand{\unflatten}{\mathsf{unflatten}}
\newcommand{\dtuple}[1]{\laa {#1} \raa}
\newcommand{\Uequiv}{\sim_{\ultrafilter}}
\newcommand{\fromIto}{{}^I\!}
\newcommand{\Uproj}{\vec{\mathbf{\pi}}}
\newenvironment{blockindent}
  {\bigskip\hfill\begin{minipage}{0.95\linewidth}}
  {\end{minipage}\bigskip}
\newcommand{\keyformula}{\mathsf{\kappa}}
\newcommand{\keysymbol}{\laa \keyformula, 3 \raa}
\newcommand{\keyrelation}{\mathsf{K}}
\newcommand{\basicaff}{A_{\vecfe{e}_0:\vec{\mathbf{e}}}}
\newcommand{\basicaffital}{A_{\vec{e}_0:\vec{\boldsymbol{e}}}}
\newcommand{\vecfe}[1]{\vec{\fe{{#1}}}}
\newcommand{\hatfe}[1]{\hat{\fe{{#1}}}}

\newcommand{\RelSet}{\Psi} 

\newcommand{\laa}{\langle}
\newcommand{\raa}{\rangle}
\newcommand{\la}{\langle}
\newcommand{\ra}{\rangle}
\newcommand{\Fm}{\mathit{Fm}}
\newcommand{\formulas}{\Fm_{\lang}}

\newcommand{\lang}{{\mathcal{L}}}
\newcommand{\langD}{\lang_{\RelSet}}
\newcommand{\cand}{\mathcal{R}\mathit{Symb}_{\lang}}
\newcommand{\AffTr}{\mathsf{AffineTrf}}

\newcommand{\minkkerning}{\hspace{-1pt}}

\newcommand{\Tr}{\mathsf{Tr}}

\newcommand{\minkst}{\mathcal{M}\minkkerning\mathit{ink}}

\newcommand{\newtst}{\mathcal{N}\!\mathit{ewt}}

\newcommand{\galst}{\mathcal{G}\mathit{al}}

\newcommand{\euclg}{\mathcal{E}\!\mathit{ucl}}

\newcommand{\cng}[1]{\bm{{ \cong}}_{#1}}

\newcommand{\llcon}{\bm{\uplambda}}

\newcommand{\Rodel}{{\mathfrak{R}}}
\newcommand{\Nodel}{{\mathfrak{N}}}
\newcommand{\Model}{{\mathfrak{M}}}

\newcommand{\origin}{\vec{\textsf{\small o}}}
\newcommand{\unit}{\vec{\textsc e}}

\newcommand{\aut}[1]{\mathsf{Aut}\; #1}
\newcommand{\autind}[1]{\widetilde{\mathsf{Aut}}\;#1}
\newcommand{\AffAut}[1]{\mathsf{AffAut}\,#1}

\newcommand{\fe}[1]{{\mathrm{#1}}}
\newcommand{\fv}[1]{{#1}}
\newcommand{\gv}[1]{{\boldsymbol{\mathit{#1}}}}
\newcommand{\gee}[1]{{\mathsf{#1}}}

\newcommand{\OAff}{\mathcal{OA}\mathit{ff}}

\newcommand{\relst}{\mathcal{R}\mathit{el}}

\newcommand{\Lclst}{\mathcal{LC}\mathit{lass}}

\newcommand{\field}{\mathfrak{F}}

\newcommand{\udfield}{{\mathfrak{F}^{\ultrafilter}}}

\newcommand{\rel}{\mathsf{R}}

\newcommand{\ultrafilter}{\mathcal{U}}
\newcommand{\proj}{\pi}
\newcommand{\F}{F}
\newcommand{\geometry}{\mathcal{G}}
\newcommand{\Col}{{\mathsf{Col}}}

\newcommand{\Bw}{{\mathsf{Bw}}}

\title{Definable coordinate geometries over fields, part 1: theory} 
\author{Judit Madar\'asz}
\address{Judit Madar\'asz, HUN-REN Alfr\'ed R\'enyi Institute of Mathematics, Budapest, Hungary}
\email{madarasz.judit@renyi.hu}

\author{Mike Stannett}
\address{Mike Stannett, School of Computer Science, The University of Sheffield, Sheffield, UK}
\email{m.stannett@sheffield.ac.uk}

\author{Gergely Sz\'ekely}
\address{Gergely Sz\'ekely, HUN-REN Alfr\'ed R\'enyi Institute of Mathematics, Budapest, Hungary  \& University of Public Service, Budapest, Hungary}
\email{szekely.gergely@renyi.hu}

\date{\today}

\thanks{This research is supported by the Hungarian National
  Research, Development and Innovation Office (NKFIH), grants
  no.\ FK-134732 and TKP2021-NVA-16.}

\begin{document}
\maketitle

\begin{abstract} 
  We define general notions of coordinate
  geometries over fields and ordered fields, and consider coordinate
  geometries that are given by finitely many relations that are
  definable over those fields. We show that the automorphism group of
  such a geometry determines the geometry up to definitional
  equivalence; moreover, if we are given two such geometries
  $\geometry$ and $\geometry'$, then the concepts (explicitly
  definable relations) of $\geometry$ are concepts of $\geometry'$
  exactly if the automorphisms of $\geometry'$ are automorphisms of
  $\geometry$.  We show this by first proving that a relation is a
  concept of $\geometry$ exactly if it is closed under the
  automorphisms of $\geometry$ and is definable over the field;
  moreover, it is enough to consider automorphisms that are affine
  transformations.
\end{abstract}


\section{Introduction}

The most famous ancient work that treats geometry systematically and
axiomatically is Euclid's Elements (\textit{c}.\ 300 B.C.).  Another
important milestone was when Descartes opened a new analytic
perspective to geometry by introducing coordinate systems.  Later
Hilbert provided a formal axiom system for Euclidean geometry, which
Tarski modified to get an axiomatization within first-order logic. A
typical representation theorem for an axiom system such as Hilbert's
or Tarski's states that all its models are isomorphic to Cartesian
spaces over certain fields, and hence coordinatizable.

In this paper, we define general notions of coordinate geometries over
fields and ordered fields. This definition naturally extends to the
notion of coordinatizable geometries, \ie models that are
isomorphic to coordinate geometries.  For example, by Tarski's theorem
\cite[Thm.1]{Tarski59}, any model of his axiom system is isomorphic to
a two-dimensional coordinate geometry $\euclg$ (see
p.\pageref{def:eucl}) over some real closed field.  These kinds of
representation theorem show the strong connection between the
synthetic (purely axiomatic) and analytic (coordinate-system based)
approaches to geometry.  In \cite{CB20}, an axiomatization for
Minkowski spacetime is given in similar style.

The coordinate geometries we consider are those given
by finitely many relations definable over fields or ordered fields. We show that
the automorphism group of such a geometry determines the geometry up
to definitional equivalence; moreover, if we are given two such
geometries $\geometry$ and $\geometry'$, then the concepts (explicitly 
definable relations) of $\geometry$ are concepts of
$\geometry'$ exactly if the automorphisms of $\geometry'$ are
automorphisms of $\geometry$, see Corollary~\ref{cor-erlangen} and
Theorem~\ref{thm-erlangen}.  We show this by first proving that a
relation is a concept of $\geometry$ exactly if it is closed under the
automorphisms of $\geometry$ and is definable over the field;
moreover, it is enough to consider automorphisms that are affine
transformations, see Theorem~\ref{THM1}. Similar investigations on the
lattices of concept-sets of various structures can be found in
\cite{SSU14,MuSem20,SemSop21,SemSop22,SemSop24}.

Motivated by questions in the philosophy of physics, there is much
ongoing research focussed on finding and analyzing criteria which can
be used to determine when one mathematical object is structurally
richer than another, see \eg
\cite{North09,SwansonHalvorson12,Barrett14,Wilhelm21,Barrett22,BMW23}.
One such criterion is (SYM*), which is formulated in
\cite{SwansonHalvorson12} as: ``If $Aut(X)$ properly contains
$Aut(Y)$, then $X$ has less structure than $Y$.'' By
Corollary~\ref{cor-erlangen} on p.\pageref{cor-erlangen}, criterion
(SYM*) works perfectly for a large class of coordinate geometries when
`structure' is understood in terms of explicitly definable
relations.

This work is part of an ongoing project aiming to understand the
algebras of concepts of concrete mathematical structures and the
conceptual `distances' between them \cite{KSzLF20,KSz21,KSz25}.  By
Theorem~\ref{THM1}, in the case of many historically relevant
coordinate geometries, one can reduce the understanding of concepts to
an understanding of (affine) automorphisms, which can be useful for
solving problems associated with that project.  For example,
Theorem~\ref{THM1} is a key result used to prove a conjecture of
Hajnal Andr\'eka stating that adding any classical concept to the
concepts of special relativity results in the theory of late classical
kinematics \cite{Hconj}.

This research is part of the Andr\'eka--N\'emeti school's project to give
a logic-based foundation and understanding of relativity theories in
the spirit of the Vienna Circle and Tarski's initiative Logic,
Methodology and Philosophy of Science \cite{AMNNSz11,Friend15,FF21}.

\subsection*{Notation and conventions}
In general, we use standard model theoretic and set theoretic
notations, and we assume that the reader is familiar with basic
algebraic constructs like groups, fields and vector spaces. The symbol
$\Box$ indicates the end (or absence) of a proof. By a \emph{model},
we mean a first-order structure in the sense of classical model
theory.

The language of \emph{fields} contains binary operations $+$ and $\cdot$ for addition and multiplication and constants $0$ and $1$; in addition to these, the language of \emph{ordered fields} contains a binary relation $\leq$ for ordering. Other common operations, \eg subtraction, are derived in the usual way.

We write $\natnumb$ for the set of positive natural numbers, and fix
an enumeration $(\fv{v}_i)_{i \in \natnumb} =
(\fv{v}_1,\fv{v}_2,\fv{v}_3,\ldots)$ of distinct variables.

To abbreviate formulas, we use equality and operations on sequences of variables componentwise, \eg if $\vec{x}=\la v_1,\ldots, v_d\ra$ and $\vec{y}=\la v_{d+1},\ldots,v_{2d}\ra$, then $\vec{x}=\vec{y}$ abbreviates formula $v_1=v_{d+1}\land \dots \land v_d=v_{2d}$ and $\vec{y}+v\vec{x}$ denotes the sequence  $\la v_{d+1}+v_{3d+1}\cdot v_{1},\dots, v_{2d}+v_{3d+1}\cdot v_{d}\ra$ of terms (we take $v$ to stand for the next fresh variable, in this case $v_{3d+1}$).

Let $\rel\subseteq H^n$ be an arbitrary relation over a nonempty set
$H$ and let $f\colon H \to H$ be a map.  We say that $\rel$ is
\emph{closed under} $f$ if{}f $(a_1,\dots, a_n)\in \rel \implies
(f(a_1),\dots, f(a_n))\in\rel$, and that $f$ \emph{respects} $\rel$
if{}f $(a_1,\dots, a_n)\in \rel \iff (f(a_1),\dots, f(a_n))\in\rel$.
If $G$ is a set of functions from $H$ to $H$ that form a group under
composition, then $\rel$ is closed under all elements of
$G$ if{}f all elements of $G$ respect $\rel$. We often write
``$\rel(a_1,\dots,a_n)$'' in place of ``$(a_1,\dots,a_n) \in \rel$''.

Given models $\Model$ and $\Nodel$ for some first-order languages, we
write $M$ and $N$ to denote their universes, respectively. If
$(a_1,\ldots,a _n)\in M^n$ and $\varphi$ is a formula in the language
of $\Model$, we write $\Model\models\varphi[a_1,\ldots,a_n]$ to mean
that $\varphi$ is satisfied in model $\Model$ by any evaluation
$e\colon\{\fv{v}_1,\fv{v}_2,\ldots\}\to M$ of variables for which
$e(\fv{v}_1)=a_1$, \ldots, $e(\fv{v}_n)=a_n$. If $\varphi$ is a
formula, the expression $\varphi(\fv{v}_1,\ldots,\fv{v}_n)$ indicates
that the free variables of $\varphi$ come from the set $\{
\fv{v}_1,\ldots,\fv{v}_n \}$.

The set of automorphisms of model $\Model$ is denoted by
$\aut{\Model}$. We note that, if the language of model $\Model$
contains only relation symbols, then function $f\colon M\to M$ is an
automorphism of $\Model$ exactly if it is a bijection that respects
all the relations of $\Model$.

The inequality $\mathfrak{A}\leq\mathfrak{B}$ denotes that algebraic
structure $\mathfrak{A}$ is a subalgebra of algebraic structure
$\mathfrak{B}$. If $\Model$ is a model and $\rel$ is a relation on the
universe of $\Model$, then model $\laa \Model,\rel\raa$ is the
expansion of $\Model$ with relation $\rel$ (to some language that
contains exactly one extra relation symbol whose interpretation in
$\Model$ is $\rel$).

If $A$ and $B$ are sets, then both $A\subsetneq B$ and $B\supsetneq A$
mean that $A$ is a proper subset
of $B$.

\section{Concepts}

We say that an $n$-ary relation $\rel$ on $M$ is \emph{definable} in
$\Model$ if{}f there is a formula $\varphi(\fv{v}_1,\ldots,\fv{v}_n)$
in the language of $\Model$ that defines it; \ie for every
$(a_1,\ldots,a_n)\in M^n$, we have
\begin{equation}\label{eq:def-def}
(a_1,\ldots,a_n) \in \rel \quad\Longleftrightarrow\quad
\Model\models\varphi[a_1,\ldots,a_n].
\end{equation}

Definable $n$-ary relations will also be called $n$-ary
\emph{concepts} of $\Model$, and the set of $n$-ary concepts of
$\Model$ will be denoted by $\can{\Model}$.  The \emph{concept-set} of
$\Model$, \ie the set of all relations definable in $\Model$) is then
\[
\ca{\Model}\defeq\bigcup_{n\in\natnumb}\can{\Model}.
\]

We call models $\Model$ and $\Nodel$ \emph{definitionally equivalent} if{}f they have the same concepts, \ie $\ca{\Model}=\ca{\Nodel}$. This notion of definitional equivalence is just a reformulation of the usual one that best fits the context of this paper, see \eg \cite[p.51]{HMT71} or \cite[p.453]{Monk2000} for the standard definition. 

In particular, because the universe $M$ of $\Model$ is definable in $\Model$ as a unary relation, we have 
  \begin{equation*}
    \ca{\Model}\subseteq \ca{\Nodel} \implies M\subseteq N,
  \end{equation*}
  and hence
  \begin{equation*}
    \ca{\Model}= \ca{\Nodel} \implies M=N.
  \end{equation*}

  The cylindric-relativized set algebra obtained from model $\Model$
is denoted by $\CA{\Model}$; see Monk~\cite{Monk2000}.  
Roughly speaking, $\CA{\Model}$ is an algebraic structure
whose universe is basically the set of concepts of $\Model$,
whose operations correspond to the logical
connectives, and among whose constants are the ones that correspond to
logical \textit{true} and \textit{false}. To understand the present paper, it is not necessary
to be familiar with the notion of cylindric algebras. We use them only occasionally, to
reformulate some of our results in a more elegant algebraic form.
We note that,  
\begin{equation}
\label{subCA}
\CA{\Model}\leq\CA{\Nodel}\   \Longleftrightarrow \ \big(\ca{\Model}\subseteq\ca{\Nodel}
\text{ and } M=N\big).
\end{equation}

\section{Coordinate geometries over fields and ordered fields}

\subsection{Definitions}

Throughout this paper all geometries will be
$d$-dimensional,\footnote{In general, it makes sense to discuss
  geometries in 1-dimension. However, the definition we use in this
  paper may not be appropriate when $d=1$ because in one-dimension the
  collinearity relation is just the universal relation (\ie any three
  points are collinear) and betweenness gives less structure on a line
  than an ordering.} for some fixed integer $d \geq 2$, and defined
over some field or ordered field $\field$ with universe $F$. The
points of coordinate space $F^d$ will be the points of our geometries.

When $\field=\la \F,+,\cdot,0,1\ra$ is a field, the ternary relation
$\Col$ of \emph{collinearity} on $\F^d$ is defined 
for all $\vec{p},\vec{q},\vec{r}\in\F^d$ by
\[
\Col(\vec{p},\vec{q},\vec{r}\,)\ \defiff\ 
\vec{q}=\vec{p}+\lambda(\vec{r}-\vec{p}\,)\ \text{ for some }\ \lambda\in\F,\ \text{ or }\ \vec{r}=\vec{p} .
\]
We call model $\la \F^d,\Col\ra$ the $d$-dimensional \emph{affine geometry} over field $\field$. 
Our notion of affine geometry is a natural generalization of the notion affine Cartesian space of \cite[Def.1.5.]{SzT79} to fields which are not ordered.

Now suppose $\field=\la \F,+,\cdot,0,1,\leq\ra$ is an ordered field.%
\footnote{That
  $\la\F,+,\cdot,0,1,\leq\ra$ is an ordered field means that 
$\la\F,+,\cdot,0,1\ra$ is
  a field which is totally ordered by $\leq$, and we have the following
  two properties for all ${x},{y},{z}\in \F$: (i) ${x}+{z}\leq 
{y}+{z}$ if 
${x}\leq {y}$,
  and (ii) $0\leq {x}{y}$ if $0\leq {x}$ and $0\leq {y}$.} 
Then the ternary relation $\Bw$ of \emph{betweenness} on $\F^d$ is defined 
for every $\vec{p},\vec{q},\vec{r}\in\F^d$ by
\[
\Bw(\vec{p},\vec{q},\vec{r}\,)\defiff 
\vec{q}=\vec{p}+\lambda(\vec{r}-\vec{p}\,)\ \text{ for some }\ \lambda\in [0,1] \subseteq \F.
\]
Of course, $\Col$ is also defined for ordered fields and is definable
in terms of $\Bw$ by:
\[
\Col(\vec{p},\vec{q},\vec{r}\,)\ \Longleftrightarrow\
\Bw(\vec{p},\vec{q},\vec{r}\,)\vee
\Bw(\vec{q},\vec{p},\vec{r}\,)\vee
\Bw(\vec{p},\vec{r},\vec{q}\,).
\]
We call model $\la \F^d,\Bw\ra$ the $d$-dimensional \emph{ordered
affine geometry} over ordered field $\field$. Our notion of ordered
affine geometry coincides with the notion of affine Cartesian space
given in \cite[Def.1.5.]{SzT79}.

\bigskip 
In much of what follows, we will need to re-express $n$-tuples over $M^d$ ($n$-tuples of ($d$-tuples over $M$)) as $(dn)$-tuples over $M$. Given the potential confusion arising from treating the same collection of values and/or variables using tuples of different lengths, we will adopt the notational convention of denoting points in $M^d$ ($d$-tuples over $M$) using arrows and angle-brackets; for example, $\vec{a} = \dtuple{a_1,\ldots,a_d}$. Other tuples will be written out in the usual way using parentheses, and overlines (the context will always be clear). Given any $d$-tuple $\vec{p}$, we denote it's $i$'th component by $p_i$, and note that $\vec{p} = \la p_1, \dots, p_d \ra$.

Given any $n$-tuple $\bar p$ over $M^d$, we write $\flatten(\bar{p})$
to denote the naturally corresponding `flattened' $(dn)$-tuple over
$M$, \ie given any $n$-tuple $\bar{p}=(\vec{p}_1, \dots, \vec{p}_n)$
of $d$-tuples $\vec{p}_1 = \dtuple{ p_{11}, \ldots, p_{1d} }$, \dots ,
$\vec{p}_n = \dtuple{ p_{n1}, \ldots, p_{nd} }$ in $M^d$, we define
\[
\flatten(\bar p) 
 \defeq 
\left( p_{11},\ldots, p_{1d}, \dots, p_{n1},\ldots, p_{nd} \right) .
\]

For convenience, we will generally write $(\hat{p}_1, \dots, \hat{p}_n)$ for $\flatten(\bar{p})$.
For example, in the case that $d = 2$, $\vec{p}_1 = \dtuple{1,2}$, $\vec{p}_2 = \dtuple{3,4}$ and $\vec{p}_3 = \dtuple{5,6}$, we have $(\hat{p}_1,\hat{p}_2) = (1,2,3,4)$ and $(\hat{p}_1, \dots, \hat{p}_3) = (1,2,3,4,5,6)$.  The ``hatted'' notation $\hat{p}$ always indicates that the underlying tuple is a $d$-tuple, namely $\vec{p}$.
 We use the analogous notation for sequences of variables.

Since $\flatten$ is a bijection from $(M^d)^n$ onto $M^{dn}$ it has an inverse (we call it $\unflatten$).
If $\rel$ is any $n$-ary relation on $M^d$, we write $\widehat{\rel}$ to denote the corresponding $(dn)$-ary relation 
on $M$, \ie for all $\bar p\in M^{dn}$,
\[
		\bar p\in\widehat{\rel} \quad\defiff\quad \unflatten(\bar p)\in\rel.
\]
Given any $\vec{p}_1, \dots, \vec{p}_n$ in $M^d$ ($n \in \natnumb$),
the bijective nature of $\flatten$ and $\unflatten$ means this can
also be written as
\begin{equation}\label{eq-flatten}
		(\hat{p}_1, \dots, \hat{p}_n)\in\widehat{\rel} \quad\iff\quad (\vec{p}_1,\dots, \vec{p}_n)\in\rel.
\end{equation}

\medskip
We say that $n$-ary relation $\rel\subseteq \left(M^d\right)^n$ is
\emph{definable over $\Model$} if{}f the corresponding relation
$\widehat{\rel}\subseteq M^{dn}$ is definable in
$\Model$. In particular, $\Col$ is definable
over $\field$ if $\field$ is a field, and $\Bw$ is definable over
$\field$ if $\field$ is an ordered field.

\begin{definition}[coordinate geometry over a field]
\label{defn:coordinatized-over-field}
We call a model $\geometry$ a ($d$-dimensional) \emph{coordinate geometry over
a field $\field=\la\F,+,\cdot, 0, 1 \ra$} if{}f the following conditions are
satisfied:
\begin{itemize}
\item the universe of $\geometry$ is $\F^d$ (called the set of
  points);
\item $\geometry$ contains no functions or 
constants;
\item the ternary relation $\Col$ of collinearity on points 
is definable in $\geometry$.
\end{itemize}
\end{definition}

\begin{definition}[coordinate geometry over an ordered field]
\label{defn:coordinatized-over-ordered-field}
We call a model $\geometry$ a ($d$-dimensional) \emph{coordinate
{geometry} over an ordered field $\field=\langle \F, +, \cdot,0,1,\leq
\rangle$} if{}f the following conditions are satisfied:
\begin{itemize}
\item the universe of $\geometry$ is $\F^d$ (the set of points);
\item $\geometry$ contains no functions or constants;
\item
the ternary relation $\Bw$ of betweenness on points is definable in
$\geometry$.
\end{itemize}
\end{definition}

\begin{definition}[field-definable and FFD coordinate geometries]
Let $\geometry$ be a coordinate geometry over some field or ordered
field $\field$. We call $\geometry$ \emph{field-definable} if{}f all
the relations of $\geometry$ are definable over $\field$; and
\emph{finitely field-definable (FFD)} if{}f it is field-definable and
it contains only finitely many relations.
\end{definition}

\section{Some examples}

Since relations $\Col$ and $\Bw$ are definable over the corresponding
(ordered) field $\field$, both affine and ordered affine geometries
are FFD coordinate geometries for all $d\ge 2$.

Euclidean geometry can be introduced as the following FFD coordinate
geometry over any ordered field $\field$:
  \begin{equation*}\label{def:eucl}
    \euclg  \ \defeq \ \la \F^d,\cong,\Bw\ra, 
  \end{equation*}
  where the 4-nary relation $\cong$ of congruence is defined by: 
  \begin{equation*}
    (\vec{p},\vec{q}\,)\,\cng{}\,(\vec{r},\vec{s}\,) \defiff  
(p_1-q_1)^2+\dots +(p_d-q_d)^d=(r_1-s_1)^2+\dots +(r_d-s_d)^2.
  \end{equation*}

The language of $\euclg$ is exactly the language used by Tarski to
axiomatize elementary Euclidean geometry, see \eg
\cite{Tarski59,SzT79,TG98}.

Notable classical geometries such as Minkowski spacetime ($\minkst$),
Galilean spacetime ($\galst$) and Newtonian spacetime ($\newtst$), as
examples for FFD coordinate geometries, can be introduced in a similar
manner, see \cite{part2}. Here, we introduce a seemingly simpler
geometry $\relst$ for relativistic spacetime, which turns out to be
definitionally equivalent to $\minkst$ (this equivalence will be shown
in \cite{part2}). We define
  \begin{equation*}
    \relst  \ \defeq \ \la \F^d, \llcon,\Bw\ra,
  \end{equation*}
  where the binary relation $\llcon$ of lightlike relatedness\footnote{Relation $\llcon$ captures whether a light signal
    can or cannot be sent between two points of spacetime. This is a
    natural basic concept to axiomatize relativity theory; it appears
    directly in \cite{Pambuccian07, Mundy} and in spirit in the
    language of James Ax's Signaling theory, see \cite{Comparing,
      Ax}. } is defined by:
  \begin{equation*}
    \vec{p}\,\llcon\,\vec{q} \ \defiff\  (p_1-q_1)^2=(p_2-q_2)^2+\ldots+(p_d-q_d)^2.
  \end{equation*}

For cardinality reasons, it is clear that there have to be coordinate
geometries which are not field-definable. A concrete example of such a
geometry over the ordered field of reals is the geometry:
  \begin{equation*}
    \la \mathbb{R}^d, \mathbb{Z}^d,\Bw\ra.
  \end{equation*}
This is because the unary relation $\mathbb{Z}^d$ on points is not
definable over the field of reals; this fact follows easily from
Tarski's theorem on quantifier elimination of real-closed fields
\cite{Tarski98}.

\section{On automorphisms of geometries}
\label{sec:geometry-automorphisms}

Echoing Klein's Erlangen program, many of the results in this paper
rely on a correspondence between the concepts of a model and those of
its natural symmetries that respect specific structural
relationships.

\subsection{Key theorems}
In this section, we state our key theorems: they confirm
that the concepts of a geometry are intimately associated with its
affine automorphism group.

\begin{definition}[Affine automorphisms]
\label{defn:affaut}
Suppose $F$ is the universe of some field or ordered field $\field$. 
A map $A\colon\F^d\rightarrow\F^d$ is called an 
\emph{affine transformation} if{}f
it is the composition of a linear bijection 
and a translation,
\ie $A=\tau\circ L$ for some translation $\tau$
and linear bijection $L$. We note that all affine
transformations are bijections.%
\footnote{
	As usual, a map $L\colon\F^d\rightarrow\F^d$ is called a \emph{linear transformation}
	if{}f
	for every $\vec{p},\vec{q}\in\F^d$ and $a\in\F$,
	$
	L(\vec{p}+\vec{q}\,)=L(\vec{p}\,)+L(\vec{q}\,)$ and 
	$L(a\vec{p}\,)=a L(\vec{p}\,)$; $L$ is
	a \emph{linear bijection} if{}f it is an invertible linear
	transformation;
	 and a map $\tau\colon\F^d\rightarrow\F^d$ is 
	a \emph{translation}
	if{}f there is $\vec{q}\in\F^d$ such that $\tau(\vec{p}\,)=\vec{p}+\vec{q}$
	for every $\vec{p}\in\F^d$.
}  

The set of affine transformations is
denoted by $\AffTr$.
\\
\indent
Affine transformations that are automorphisms of a coordinate
geometry $\geometry$ over $\field$ are called
\emph{affine automorphisms} of $\geometry$, and the set of these is denoted by
$\AffAut{\geometry}$. Thus
\begin{equation}\label{eq-def:AffAut}
\AffAut{\geometry} \defeq \aut{\geometry}\cap\AffTr.
\end{equation}
If $\geometry$ is an affine or ordered affine geometry, we have $\AffAut{\geometry} = \AffTr$.
\end{definition}

\begin{restatable}{theorem}{THMone}
\label{THM1}
Let $\field$ be an ordered field or a field that has more than two
elements, let $\geometry$ be an FFD coordinate geometry over $\field$,
and let $\rel$ be a relation on points of $\F^d$. Then the following
statements are equivalent.
\begin{itemize} 
\item[(i)] $\rel$ is a concept of $\geometry$ (\ie
$\rel$ is definable in $\geometry$).
\item[(ii)] $\rel$ is definable over $\field$ and is closed under automorphisms
of $\geometry$.
\item[(iii)] $\rel$ is definable over $\field$ and is closed under affine
automorphisms of $\geometry$. 
\end{itemize}
\end{restatable}

Proving this theorem requires that we first develop a number of
supporting lemmas, and completing the proof is consequently relegated
to Section \ref{THM1-proof} (p.\pageref{THM1-proof}, below).  But
first, we state and prove various corollaries of Theorem~\ref{THM1},
showing that the subset-inclusion poset of the concept-sets associated
with FFD coordinate geometries is dually isomorphic to the
subgroup-inclusion poset of their respective automorphism groups.

We will apply these results in Part 2 of this study \cite{part2} to
compare the sets of concepts of several historically significant
spacetimes by determining their affine automorphism groups.

\begin{remark}
Let $\Model$ and $\Nodel$ be two models. Clearly,
\[
\ca{\Model}=\ca{\Nodel}\ \Longrightarrow\ \aut{\Model}=\aut{\Nodel},
\]
because definable relations must be closed under automorphisms.
Furthermore, if the universes of $\Model$ and $\Nodel$ are the same, then
\[
\ca{\Model}\subseteq\ca{\Nodel}\ \Longrightarrow\ 
\aut{\Model}\supseteq\aut{\Nodel}.
\]

The reverse implication ($\Longleftarrow$) does not hold in general.
For example, if $\Rodel$ is the field of reals and $\Rodel'$ is the
expansion of $\Rodel$ with the unary relation $N$ of being a natural
number, then $\aut{\Rodel} = \aut{\Rodel'}$ because both $\Rodel$ and
$\Rodel'$ have no nontrivial automorphisms, but
$\ca{\Rodel'}\not\subseteq\ca{\Rodel}$ because the set of natural
numbers is not definable in the field of real numbers.  Nonetheless,
this reverse implication \emph{does} hold for FFD coordinate
geometries -- as we now show.
\end{remark}

\begin{theorem}
\label{thm-erlangen}
Assume that $\field$ is an ordered field or a field with more than two elements, 
and that $\geometry$ and $\geometry'$ are
FFD coordinate geometries over
$\field$. Then:
\begin{enumerate}[(i)]
\item
$\ca{\geometry}\subseteq\ca{\geometry'}\ \Longleftrightarrow\
\aut{\geometry}\supseteq\aut{\geometry'}$.
\item
\label{Er2}
$\ca{\geometry}\subseteq\ca{\geometry'}\ \Longleftrightarrow\
 \AffAut{\geometry}\supseteq
\AffAut{\geometry'}$. 
\end{enumerate}
\end{theorem}

\begin{proof}
Note that $\geometry$ and $\geometry'$ have the same universe. So, as
we have already seen, the $\Longrightarrow$ implications of both items
hold.  Thus we have to prove
only the $\Longleftarrow$ directions.

(i): Assume that $\aut{\geometry}\supseteq\aut{\geometry'}$ and that
$\bm{\upvarrho}\in\ca{\geometry}$.  Then, by Theorem~\ref{THM1},
$\bm{\upvarrho}$ is definable over $\field$ and is closed under
automorphisms of $\geometry$. Then $\bm{\upvarrho}$ is closed under
automorphisms of $\geometry'$ since
$\aut{\geometry}\supseteq\aut{\geometry'}$. And so, by
Theorem~\ref{THM1} again, $\bm{\upvarrho}$ is definable in
$\geometry'$, \ie $\bm{\upvarrho}\in\ca{\geometry'}$.

(ii): The proof mirrors that of (i); the only
difference is that here one has to refer to affine automorphisms
instead of automorphisms.
\end{proof}

\begin{corollary}
\label{cor-erlangen}
Assume that $\field$ is an ordered field or a field that has more than two elements, and that $\geometry$ and $\geometry'$ are FFD coordinate geometries over
$\field$. Then:
\begin{enumerate}[(i)]
\item  $\ca{\geometry}=\ca{\geometry'}\ \Longleftrightarrow\
\aut{\geometry}=\aut{\geometry'}$.
\item 
\label{ErCor2}
$\ca{\geometry}=\ca{\geometry'}\ \Longleftrightarrow\
\AffAut{\geometry}=\AffAut{\geometry'}$.
\item  $\ca{\geometry}\subsetneq\ca{\geometry'}\ \Longleftrightarrow\
\aut{\geometry}\supsetneq\aut{\geometry'}$.
\item $\ca{\geometry}\subsetneq\ca{\geometry'}\ \Longleftrightarrow\
\AffAut{\geometry}\supsetneq\AffAut{\geometry'}$. \hfill $\Box$
\end{enumerate}
\end{corollary}

Let us note that $\la\aut{\Model},\circ\ra$ is a group for every model
$\Model$, and similarly $\la\AffAut{\geometry},\circ\ra$ is a group
for every coordinate geometry $\geometry$, where $\circ$ is the
standard composition of functions. By \eqref{subCA}, the following is
a reformulation of Theorem~\ref{thm-erlangen} in terms of cylindric
algebras.

\begin{corollary}
\label{main-corollary}
Assume that $\field$ is an ordered field or a field that has more than two elements, and that $\geometry$ and $\geometry'$ are
FFD coordinate geometries over $\field$. Then:
\begin{itemize}
\item[(i)] $\CA{\geometry} \leq
  \CA{\geometry'}\ \Longleftrightarrow\ \la\aut{\geometry},\circ\ra\geq
  \la\aut{\geometry'},\circ\ra$.
\item[(ii)] $\CA{\geometry} \leq
  \CA{\geometry'}\ \Longleftrightarrow\ \la\AffAut{\geometry},\circ\ra\geq
  \la\AffAut{\geometry'},\circ\ra$. \hfill $\Box$
\end{itemize} 
\end{corollary}

As observed above, proving (and fully explaining) Theorem~\ref{THM1}
requires that we first consider a number of preliminary results; the
full proof concludes in Section~\ref{THM1-proof}.

\subsection{Automorphisms and affine automorphisms}
\label{autonotations}

Suppose $F$ is the universe of some field or ordered field $\field$. 
Given a function $f\colon\F\rightarrow\F$, we write
$\widetilde{f}\colon\F^d\rightarrow\F^d$ for the map induced on $F^d$ by $f$ componentwise, i.e.\
 $\widetilde{f}\bigl(\laa p_1,\ldots,p_d\raa\bigr)\defeq\laa 
f(p_1),\ldots,f(p_d)\raa$.
We note that if $f$ is a bijection, then so is $\widetilde{f}$,
and we define the induced automorphisms on $F^d$ to be the members of the set
\[
\autind{\field}\defeq \{\widetilde{\alpha} :
\alpha\in\aut{\field}\}.
\]

\begin{lemma}[Fundamental Theorem of Affine Geometry] 
\label{fundamental-thm}
${}$
\begin{itemize}
\item[(i)] Let $\field=\la \F,+,\cdot,0,1\ra$ be a field that has more than two elements. 
Then
\[
\aut{\la \F^d,\Col\ra}= \AffTr\circ\autind{\field}.
\]
\item[(ii)] Let $\field=\la \F,+,\cdot,0,1,\leq\ra$ be an ordered field.
Then
\[
\aut{\la \F^d,\Bw\ra}= \AffTr\circ\autind{\field}.
\] 
\end{itemize}
\end{lemma} 
\begin{proof}
\underline{(i)}:
It is easy to check that maps induced by automorphisms of $\field$, as well as
affine transformations, are 
automorphisms of $\la \F^d,\Col\ra$, whence so are their compositions. This shows that
$
\AffTr \; \circ \; \autind{\field} \subseteq \aut{\la \F^d,\Col\ra}.
$
For the converse (that every automorphism is a composition of a map induced by
an automorphism of $\field$ and an affine transformation) see, \eg \ Berger~\cite[Thm.2.6.3, p.52]{Berger} and
Tarrida~\cite[Thm.2.46, p.81]{Tarrida}.
\underline{(ii)}: This follows easily from (i) because a map induced by a field automorphism respects $\Bw$ if{}f it respects the order.
\end{proof}

\begin{prop}
\label{autlemma} 
Assume that $\field$ is an ordered field or a field that has more than
two elements, and that $\geometry$ is a field-definable coordinate
geometry over $\field$. Then \begin{equation*}
  \aut{\geometry}=\AffAut{\geometry}\;\circ\;\autind{\field}.
\end{equation*}
Moreover, the decomposition of an automorphism of $\geometry$
  into an affine automorphism following a map induced by an automorphism
  of $\field$ is unique.
\end{prop}
\begin{proof} 
Let us first note that 
$\autind{\field}\subseteq\aut{\geometry}$ 
since every relation of $\geometry$ is definable over $\field$.
Hence, by definition \eqref{eq-def:AffAut},
\[
    \AffAut{\geometry}\circ\autind{\field}
		\subseteq
		\aut{\geometry}\circ\autind{\field}
		\subseteq 
    \aut{\geometry}\circ\aut{\geometry} 
    = 
    \aut{\geometry}. 
\]

  The converse inclusion follows from Lemma~\ref{fundamental-thm},
  easily as follows. Since the relevant relation, $\Bw_\field$ or
  $\Col_\field$, is definable in $\geometry$ (depending on whether
  $\field$ is an ordered field or just a field), this relation is
  respected by the elements of $\aut{\geometry}$. Thus
  $\aut{\geometry}\subseteq\AffTr\circ\autind{\field}$ by
  Lemma~\ref{fundamental-thm}. Hence every $g\in\aut{\geometry}$ can
  be written as $g=A\circ \widetilde{\alpha}$ for some $A\in \AffTr$ and
  $\widetilde{\alpha}\in\autind{\field}$. Since $\autind{\field}\subseteq
  \aut{\geometry}$ and $\langle \aut{\geometry},\circ \rangle$ is a
  group, we have that $A\in\aut{\geometry}$. So $A\in
  \AffAut{\geometry}$, and hence $\aut{\geometry}\subseteq
  \AffAut{\geometry}\circ\autind{\field}$.

  To prove the uniqueness of the decomposition, let $A_1\circ
  \widetilde{\alpha_1}=A_2\circ \widetilde{\alpha_2}$ be an
  automorphism of $\geometry$ for some $A_1,A_2\in\AffAut{\geometry}$
  and $\alpha_1,\alpha_2\in\aut{\field}$. Then $A_2^{-1}\circ A_1
  =\widetilde{\alpha_2}\circ\widetilde{\alpha_1}^{-1}$, from which the
  uniqueness follows since the identity map is the only affine
  transformation that is also an induced automorphism on
  $\F^d$. 
\end{proof}

\subsection{Definability of relations on $d$-tuples in arbitrary models}
\label{relations_on_dtuples}
With geometries as our main application in mind, we now make
  general observations that are useful when investigating the
  definability of relations on $d$-tuples in arbitrary models.  
	
	Let
$\lang$ be an arbitrary first-order language and let $\formulas$ be
the set of first-order formulas in the language $\mathcal L$.  Then we define the following set of relation symbols for $\mathcal L$-definable relations of $d$-tuples:
\begin{equation*}\cand
\defeq \left\{\, \la\varrho,n\ra\in\formulas\times\natnumb\: :\: 
\text{the free variables of $\varrho$ are in }
  \{\fv{v}_1,\ldots,\fv{v}_{dn}\}\, \right\}.
\end{equation*}

Let $\Model$ be an arbitrary model for language $\mathcal{L}$ with universe $M$. For each relation symbol $R=\la
\varrho,n\ra\in\cand$, we define a corresponding $n$-ary relation $\RMd$ on
$M^d$ by
\begin{equation}\label{eq:def-Rf}
(\vec{p}_{1},\ldots, \vec{p}_{n})\in \RMd
	\quad\defiff\quad 
\Model\models \varrho [ \hat{{p}}_{1},\dots,
\hat{{p}}_{n} ] .
\end{equation}

\begin{remark}\label{rem-clearly}
  Clearly, $\RMd$ is definable over $\Model$ for every
  $R\in\cand$. Conversely, given any relation $\rel$ on $M^d$ that is
  definable over $\Model$, there is some $R\in\cand$ for which $\rel =
  \RMd$. For suppose $\varrho$ is a formula defining $\rel$ over
  $\Model$, and let $n$ be $\rel$'s arity. Taking $R = \langle
  \varrho,n\rangle$ then yields $\rel = \RMd$.  Thus, given any
  relation $\rel$ on $M^d$, it is definable over $\Model$ if{}f
  $\rel=\RMd$ for some $R\in\cand$.
\end{remark}

 For every $\RelSet\subseteq \cand$, let
$\langD$ denote the first-order language with relation symbols in
$\RelSet$ such that the arity of each $R=\laa \varrho,n\raa\in\RelSet$ is
$n$.

In particular, we define the following two formulas $\gamma$ and $\beta$, the 
first in the languages of fields and the second in that of ordered fields, and each having $3d$ 
free variables, $v_1,\ldots,v_{3d}$. Given
$\vec{x}=\la v_1,\ldots, v_d\ra$, $\vec{y}=\la v_{d+1},\ldots,v_{2d}\ra$,
 $\vec{z}=\la v_{2d+1},\ldots, v_{3d}\ra$ and $v=v_{3d+1}$, we define
\begin{align*}
\gamma(\hat{x},\hat{y},\hat{z}) & :=  \exists v\, 
\left(\vec{y}+ v\vec{x}=\vec{x}+v\vec{z}\ \vee\ 
\vec{z}=\vec{x}\,\right) \text{ and}\\  
\beta(\hat{x},\hat{y},\hat{z}) & := 
\exists v\, 
\left(\vec{y}+v\vec{x}=\vec{x}+ v\vec{z}\ \wedge\
0\leq v \leq 1\right). 
\end{align*}

\begin{remark}
\label{rem:gammaCol-betaBw}
Suppose $\lang$ is the language of $\field$. Then
\begin{itemize}
\item  
 $\laa \gamma,3\raa\in\cand$ and if $\field$ is a field, then
 $\laa \gamma,3\raa_{\field d}$  coincides with $\Col$ of the affine geometry
over $\field$.  
\item  $\laa \beta,3\raa\in\cand$ 
and, if $\field$ is an ordered field, then
 $\laa \beta,3\raa_{\field d}$ is $\Bw$
of the ordered affine geometry over $\field$.
\end{itemize}
\end{remark}

Throughout the remainder of this paper, we consider two cases in parallel: the case where $\field$ is a field, $\lang$ is the language of fields, $\gamma$ is the relevant formula and $\Col$ the relevant relation; and the case where $\field$ is an ordered field, $\lang$ is the language of ordered fields, $\beta$ is the relevant formula and $\Bw$ the relevant relation. To simplify the presentation of results below, we will write $\keyformula$ (`key formula') in place of the formulas $\gamma$ and $\beta$, \ie
\[
	\keyformula \quad =
	\begin{cases}
	    \quad \gamma  & \text{ if the language of $\field$ is the language of fields; and } \\
			\quad \beta   & \text{ if the language of $\field$ is the language of ordered fields}.
	\end{cases}
\]
We likewise define the `key relation' to be $\keyrelation = \keysymbol_{\field d}$, and note that $\keyrelation$ coincides either with $\Col$ or $\Bw$ depending on the language of $\field$ (cf.~Remark~\ref{rem:gammaCol-betaBw}). In both cases, we continue to write $F$ for the universe of $\field$ and note that the language of $\field$ is usually clear from context.
\medskip

For every $\RelSet\subseteq\cand$, we define a model
$\geometry_\RelSet(\Model)$ for language $\langD$ to be the model with universe $M^d$ in which the interpretation of every
$R\in\RelSet$ is $\RMd$, \ie 
\begin{equation}\label{eq:def-Gdelta}
\geometry_\RelSet(\Model)=\la M^d,\RMd\ra_{R\in\RelSet}.
\end{equation}

To express that two models $\Model$ and $\Nodel$
for possibly different languages
are the same, we introduce the following notation: we write
$\Model\doteq\Nodel$
if{}f their universes are the same
and there is a one-to-one correspondence
between their languages such that the corresponding relation symbols
and function symbols 
have the same interpretations in the two models.

\begin{remark}\label{rem:concrepetition}
  In some models of the form $\geometry_\RelSet(\Model)$, the same
  concept appears more than once, but named using different relation
  symbols. This is because, if $\varphi$ and $\psi$ are different but
  logically equivalent formulas, then the symbols $\la \varphi, n\ra$
  and $\la \psi, n\ra$ of $\cand$ differ but the corresponding
  relations are the same.  One extreme case is the conceptually
  richest model $\geometry_{\cand}(\Model)$, where every
  field-definable concept appears as a primitive concept infinitely
  many times.
\end{remark}

\begin{remark}\label{prop-doteq}
Suppose that $\lang$ is the language of $\field$. By Definitions
\ref{defn:coordinatized-over-field} and
\ref{defn:coordinatized-over-ordered-field}, model $\geometry$ is a
field-definable coordinate geometry over $\field$ if{}f there is some
$\RelSet\subseteq\cand$ such that
$\geometry\doteq\geometry_\RelSet(\field)$ and the key relation
$\keyrelation$ is definable in $\geometry$. Moreover, $\geometry$ is
FFD if{}f $\RelSet$ can be chosen finite.
\end{remark}

\begin{remark}\label{prop:eeq}
  For arbitrary models $\Model_1$ and $\Model_2$, and 
(a possibly
  infinite) subset $\RelSet$ of $\cand$, we have that
  $\geometry_{\RelSet}(\Model_1)$ and $\geometry_{\RelSet}(\Model_2)$
  are elementarily equivalent whenever $\Model_1$ and $\Model_2$ are
  elementarily equivalent. This is because, using the formula parts of the symbols of $\RelSet$, one can define a translation $\Tr$ from the common language of $\geometry_{\RelSet}(\Model_1)$ and $\geometry_{\RelSet}(\Model_2)$ to the common language of $\Model_1$ and $\Model_2$, and it can be shown by formula induction that 
  \begin{equation*}
      \geometry_{\RelSet}(\Model)\models \sigma \iff \Model \models \Tr(\sigma)  
  \end{equation*}
  holds for every model $\Model$ and sentence $\sigma$ of common language of $\Model_1$ and $\Model_2$.
\end{remark}

\subsection{Extending automorphisms from one geometry to another}

In this section, we establish the conditions under which affine
automorphisms of certain coordinate geometries are also automorphisms
of their finite extensions. In particular, if
$\geometry_{\RelSet}(\field)$ is a coordinate geometry for some finite
$\RelSet\subseteq\cand$, and $R\in\cand$, Lemma~\ref{affaut-lemma}
establishes conditions under which every affine automorphism of
$\geometry_{\RelSet}(\field)$ must also be an affine automorphism of
$\geometry_{\RelSet\cup\{R\}}(\field)$.

\begin{definition}
\label{unitvectors}
We write $\origin$ to denote the \emph{origin} $\laa 0, \dots, 0 \raa \in F^d$, and
$\unit_1,\ldots,\unit_d$ to denote the \emph{unit vectors} $\laa
1,0,\ldots,0\raa, \ldots, \laa 0,\ldots,0,1\raa \in F^d$, respectively.
\end{definition}

  The following simple proposition is well-known from linear algebra. We
  restate it here for convenience and to fix notation, and then re-express 
  the concepts embedded within it in a form suitable for our purposes.

\begin{prop}
\label{affinszemlelet}
  Suppose
$\vecfe{e}_0, \vecfe{e}_1, \ldots,\vecfe{e}_d\in \F^d$.  Then
there is an affine transformation that takes the origin $\origin$ to
$\vecfe{e}_{0}$ and each unit vector $\unit_j$ to the corresponding
$\vecfe{e}_j$ if{}f the vectors $(\vecfe{e}_1-\vecfe{e}_0$), $\ldots$,
$(\vecfe{e}_d-\vecfe{e}_0)$ are linearly independent.  Moreover, this
affine transformation is unique. \hfill $\Box$
\end{prop}

\begin{definition}
Suppose $\vecfe{e}_0 \in F^d$ and $\vec{\mathbf{e}} = ( \vecfe{e}_{1},\ldots,\vecfe{e}_{d} )$ where each $\vecfe{e}_{j} \in F^d$ for $1 \le j  \le d$. Provided the vectors $(\vecfe{e}_1-\vecfe{e}_0)$, \ldots, 
$(\vecfe{e}_d-\vecfe{e}_0)$ are linearly independent, we write 
$\basicaff$ for the unique
affine transformation of Proposition~\ref{affinszemlelet} that maps $\origin$ to $\vecfe{e}_0$ and each $\unit_j$ to the corresponding $\vecfe{e}_j$.
\end{definition}

We now introduce various formulas, in the language $\lang$ of $\field$, needed to capture the concepts expressed in
Proposition~\ref{affinszemlelet}.\footnote{Technically, we should write these using variable names $v_1, v_2, \dots$. However, to simplify our formulas and make them more intelligible, we will use metavariables such as $\vv{x}{1}$ and $\hat{x}_0$ for the various $v_n$'s and sequences thereof.} These are:

\medskip
\begin{itemize}

\item[$\iota$:]
Formula $\iota\in\formulas$ has
$d^2+d$ free variables and expresses the idea that vectors
$(\vv{x}{1}-\vv{x}{0}),\ldots,(\vv{x}{d}-\vv{x}{0})$ are linearly
independent:

\medskip
  $\begin{aligned} \iota(\hat{x}_0,\hat{x}_{1}, & \ldots,\hat{x}_{d})
  \defeq \\ & \left[ \forall \lambda_1\ldots\forall \lambda_d \left(
    \textstyle\sum_{i=1}^d\,\lambda_i(\vv{x}{i}-\vv{x}{0})=0\ \rightarrow\ \textstyle\bigwedge_{i=1}^d
    \lambda_i=0 \right) \right].
\end{aligned}$
\medskip

\item[$\theta$:] Taken together with $\iota$, formula
  $\theta\in\formulas$ (with $d^2+3d$ free variables) expresses the
  idea that affine map $\basicaffital$ takes vector $\vec{x}=\la
  x_1,\ldots,x_d\ra$ to vector $\vec{y}=\la y_1,\ldots,y_d\ra$:

\medskip
\begin{center}$
  \theta(\hat{e}_{0},\hat{e}_{1},\ldots,\hat{e}_{d},\hat{x},\hat{y})
  \defeq
     \left[ \vec{y}=\vv{e}{0}+\textstyle\sum_{i=1}^d\, x_i\cdot(\vv{e}{i}-\vv{e}{0})\right].
$\end{center}
\medskip

\item[$\theta_R$:]
For each $R=\laa \varrho,n\raa \in\cand$, we define a formula $\theta_{R}$
which expresses\footnotemark\ in the language of $\field$ that $\basicaffital$ 
exists and respects $R$:

\medskip
\(
  \begin{aligned}
  \theta_{R} & (\hat{e}_{0},\hat{e}_{1},\ldots,\hat{e}_{d}) 
               \ \defeq \ \iota(\hat{e}_0,\hat{e}_1, \ldots, \hat{e}_{d}) \\
             &  \land \ 
                  \left[ \ \forall \hat{x}_{1}\ldots\forall \hat{x}_{n}\forall\hat{y}_{1}\ldots\forall\hat{y}_{n}   \right. \\
             & \qquad \left. \left( \left( \textstyle\bigwedge_{i=1}^{n}\,
    \theta(\hat{e}_{0}, \hat{e}_{1}, \ldots, \hat{e}_{d}, \hat{x}_{i}, \hat{y}_{i}) \right) 
                  \rightarrow\ (\varrho(\hat{x}_{1}, \ldots, \hat{x}_{n})\ \leftrightarrow\
                  \varrho(\hat{y}_{1}, \ldots, \hat{y}_{n}))\right) \ \right].
  \end{aligned}
\)
\medskip

\footnotetext{We use Tarskian substitution of variables as follows: for formula $\varphi(v_1,\dots,v_n)$ and $x_1,\ldots, x_n$ metavariables, we define 
$\varphi(x_1,\dots,x_n)\ \defeq \   \exists v_1 (v_1=x_1 \land \dots (\exists v_n (v_n=x_n\land\varphi(v_1,\dots,v_n))\dots)$.
}

\item[$\theta_\RelSet$:]
Finally, if $\RelSet$ is a finite subset of $\cand$,
we define

\medskip
\begin{center}$
  \theta_\RelSet(\hat{e}_{0},\hat{e}_{1},\ldots,\hat{e}_{d})
  \ \defeq \ 
    \bigwedge_{R\in\RelSet}\,
    \theta_{R}(\hat{e}_{0},\hat{e}_{1},\ldots,\hat{e}_{d}).
$\end{center}
\medskip

\end{itemize}
\medskip

\begin{lemma}\label{lem-theta}

Suppose
 $R\in\cand$
and let $\RelSet$ be a finite subset of $\cand$, where $\lang$ is the language of $\field$.  
Suppose $\vecfe{e}_0, \vecfe{e}_1, \ldots, \vecfe{e}_d
\in
\F^d$. Then 
\begin{itemize}
\item[(i)] $\field \models \theta_R[\hatfe{e}_0, \hatfe{e}_1, \ldots,  \hatfe{e}_d
]$
  if{}f $(\vecfe{e}_1-\vecfe{e}_0),\ldots,
  (\vecfe{e}_d-\vecfe{e}_0)$ are linearly independent and $\basicaff$
  respects $\RFd$.
\item[(ii)] $\field \models \theta_\RelSet[\hatfe{e}_0, \hatfe{e}_1, \ldots, \hatfe{e}_d
  ]$ if{}f
  $(\vecfe{e}_1-\vecfe{e}_0),\ldots, (\vecfe{e}_d-\vecfe{e}_0)$ are
  linearly independent and $\basicaff$ is an affine automorphism of
  the FFD coordinate geometry
  $\geometry_{\RelSet}(\field)$.
\end{itemize}
\end{lemma} 
\begin{proof}

\underline{(i)} Formula $\iota$ captures the independence of the
corresponding vectors in terms of coordinates. Similarly, in terms of
coordinates in $\field$, formula $\theta$ uniquely defines the affine
transformation $\basicaff$. By \eqref{eq:def-Rf}, the definition of
$\RFd$, we have that $\RFd$ is respected by the transformation defined
by $\theta$ exactly if the second part of formula $\theta_R$ holds.

\underline{(ii)} This follows
from (i) because the transformation defined by $\theta$ is an automorphism of $\geometry_{\RelSet}(\field)$ exactly if it respects all of the relations $\RFd$ corresponding to some $R\in\RelSet$.
\end{proof}

\begin{lemma} 
\label{affaut-lemma}
Let $\geometry_{\RelSet}(\field)$ be a coordinate geometry for
some finite $\RelSet\subseteq\cand$ and suppose $R\in\cand$. Then
$\field\models\theta_{\RelSet}\rightarrow\theta_R$ if{}f
every affine automorphism of $\geometry_{\RelSet}(\field)$ is an affine
automorphism of $\geometry_{\RelSet\cup\{R\}}(\field)$. 
\end{lemma} 
\begin{proof}
Follows from  Lemma~\ref{lem-theta}.
\end{proof}

\subsection{Results obtained using ultrapowers}
The next theorem and lemma concern the ultrapowers of a model $\Model$, and will be used below to 
complete our proof of Theorem~\ref{THM1}.
For the benefit of readers who are unfamiliar with this concept, and also to fix our notation,
we briefly recall the following definitions. For a more detailed account, see, \eg \cite{Ho93}.

Suppose that $M$ is the universe of a model $\Model$ of some FOL language $\lang$, and that $\ultrafilter$
is an ultrafilter on some set $I$ (called an \emph{index set}); that is, $\ultrafilter$ is a collection of subsets of $I$ satisfying 
(i) $\varnothing \not\in \ultrafilter$; 
(ii) whenever
$S_1$ and $S_2$ are members of $\ultrafilter$, so is $S_1 \cap S_2$; 
(iii) whenever
$S \in	\ultrafilter$ and $S \subseteq T \subseteq I$, then $T \in \ultrafilter$; and
(iv) given any $S \subseteq I$, either $S$ or $I \setminus S$ is in $\ultrafilter$.
One can use $\ultrafilter$ to define an equivalence relation $\Uequiv$ on the collection $\fromIto M$ of 
all functions from $I$ to $M$ by defining, for all $f_1, f_2\colon I \to M$,
\[
	f_1 \Uequiv f_2  \quad\defiff\quad  \{ i : f_1(i) = f_2(i) \} \in	\ultrafilter .
\]
Given any $f\colon I \to M$, we write $f/\ultrafilter$ for the $\Uequiv$-equivalence class containing $f$, and $\fromIto M / \ultrafilter$ for the set of all such equivalence classes;

The \emph{ultrapower} of $\Model$ according to $\ultrafilter$ (written $\Model^\ultrafilter$) is the model of $\lang$ whose universe is the set
of $\Uequiv$-equivalence classes, in which the interpretation $\rel^\ultrafilter$ of any ($n$-ary) relation symbol $R$ is defined, for  $f_1, \dots, f_n \in \fromIto M$, by
\begin{equation}
\label{ultradef}
		(f_1/\ultrafilter, \dots, f_n/\ultrafilter) \in	\rel^\ultrafilter \quad\defiff\quad  
		   \{ i : ( f_1(i), \dots, f_n(i) ) \in	 \rel \} \in \ultrafilter ,
\end{equation}
where $\rel$ is the interpretation of $R$ in $\Model$. The
interpretations of functions symbols and constants are defined
analogously. Given any ($n$-ary) formula $\phi$, {\L}os's Theorem
\cite[Thm.~8.5.1]{Ho93} tells us that
\begin{equation}
\label{Los} 
  \Model^\ultrafilter \models \phi[ f_1/\ultrafilter, \dots, f_n/\ultrafilter]
	   \quad \iff \quad
	\{ i : \Model \models \phi[ f_1(i), \dots, f_n(i) ] \} \in \ultrafilter .
\end{equation}

An easy but important consequence of \eqref{ultradef} and \eqref{Los} is that whenever $\field$ is a field or ordered field, so is $\field^\ultrafilter$.

The following theorem is due to Andr\'as Simon:\footnote{Here we give 
an independent proof. Simon's proof (personal communication) uses Keisler--Shelah's theorem applied twice.}

\begin{theorem}
\label{simon-tetel}
Let $\Model$ be a FOL model and $\rel$ a relation on its universe.  
Then $\rel$ is definable in $\Model$ if{}f, for every
ultrapower $\la \Model^\ultrafilter,\rel^\ultrafilter\ra$ of $\la
\Model, \rel\ra$, every automorphism of $\Model^\ultrafilter$ is an
automorphism of $\la \Model^\ultrafilter,\rel^\ultrafilter\ra$.
\end{theorem}

\begin{proof}
By \cite[Cor.1 on p.969]{SvenoniusThm}, if $\rel$ is not definable in
$\Model$, then there is a model $\la \Nodel, S\ra$ elementarily
equivalent to $\la \Model, \rel\ra$ such that relation $S$ is not
respected by some automorphism $f$ of $\Nodel$. By the Keisler--Shelah
isomorphism theorem, there is an ultrafilter $\ultrafilter$ such that
$\la \Model^\ultrafilter,\rel^\ultrafilter\ra$ is isomorphic to $\la
\Nodel^\ultrafilter,S^\ultrafilter\ra$. The automorphism $f$ of
$\Nodel$ extends componentwise to an automorphism of
$\Nodel^\ultrafilter$ which does not respect $S^\ultrafilter$, and
hence it is not an automorphism of $\la \Nodel^\ultrafilter,
S^\ultrafilter\ra$.  Since $\la \Nodel^\ultrafilter,
S^\ultrafilter\ra$ is isomorphic to $\la
\Model^\ultrafilter,\rel^\ultrafilter\ra$, there must also be an
automorphism of $\Model^\ultrafilter$ that is not an automorphism of
$\la \mathfrak{M}^\ultrafilter,\rel^\ultrafilter\ra$.
The other direction follows immediately because automorphisms have
to respect definable relations in every model.
\end{proof}

\begin{remark}\label{rem:fin}
The fact that in finite structures a relation is definable if{}f it is respected by all automorphisms of the structure is a simple corollary of Theorem~\ref{simon-tetel} since any ultrapower of a finite structure is isomorphic to the original structure.
\end{remark}

\begin{definition}
  If $M$ is some set and $j\leq d$ ($j \in \natnumb$), we write
  $\proj_j$ for the $j$-th projection function from $M^d$ to $M$, \ie
  $ \proj_j\colon\laa p_1,\ldots, p_j,\ldots, p_d\raa \mapsto p_j.  $
  Given a set $M$ and an ultrafilter $\ultrafilter$ over some index
  set $I$, we define the function $\Uproj_{M \ultrafilter} \colon
  (M^d)^\ultrafilter \to (M^\ultrafilter)^d$ by
\[
\Uproj_{M \ultrafilter}(f/\ultrafilter) 
	\defeq \laa (\proj_1\circ f)/\ultrafilter,
\ldots, (\proj_d\circ f)/\ultrafilter\raa
\]
for all $f\colon I \to M^d$.
Since the context is always clear where we use this function below, we will generally abbreviate $\Uproj_{M \ultrafilter}$ to $\Uproj$ in what follows.
\end{definition}

\begin{lemma}
\label{lemma1}
Let $\Model$ be a model for some first-order language $\lang$, let $M$
be its universe, let $\ultrafilter$ be an ultrafilter over some index
set $I$, and $\RelSet\subseteq \cand$.  Then,
\begin{itemize}
\item[(i)] $\Uproj$ is well-defined; and
\item[(ii)] $\Uproj$ is an isomorphism from 
$\geometry_\RelSet(\Model)^{\,\ultrafilter}$ to 
$\geometry_\RelSet\left(\Model^{\,\ultrafilter}\right)$, and hence these models are isomorphic.  
\end{itemize}
\end{lemma}

\begin{proof}
\underline{(i)}
To see that $\Uproj$ is well-defined, assume that
$\gee{p}/\ultrafilter=\gee{q}/\ultrafilter$, where $\gee{p}, \gee{q}
\in \fromIto{(M^d)}$. 
Then
\[
 \{i : \gee{p}(i)=\gee{q}(i)\}\in\ultrafilter.
\]
Clearly, for every $j\leq d$,
\[
\{i : \gee{p}(i)=\gee{q}(i)\}\subseteq\{i : \proj_j\circ \gee{p}(i)
=\proj_j\circ\gee{q}(i)\}.
\]
Therefore, $\{i : \proj_j\circ \gee{p}(i)
=\proj_j\circ\gee{q}(i)\}\in\ultrafilter$, and so 
$(\proj_j\circ\gee{p})/\ultrafilter
=(\proj_j\circ \gee{q})/\ultrafilter$ for every $j\leq d$. Hence,
$\Uproj$ is well-defined.

\underline{(ii)} First, we prove that $\Uproj$ is both surjective and injective, and hence a bijection.

To see that $\Uproj$ is surjective, let $p_1\colon I\rightarrow M$, \ldots, 
${p}_d\colon I\rightarrow M$. We want to prove that
$\laa {p}_1/\ultrafilter,\ldots,{p}_d/\ultrafilter\raa$ is in the range of $\Uproj$.
Let $\gee{p}\colon I\rightarrow M^d$ be defined by 
$\gee{p}(i):=\laa {p}_1(i),\ldots,{p}_d(i)\raa$. Clearly, 
$\Uproj(\gee{p}/\ultrafilter)=\laa
{p}_1/\ultrafilter,\ldots,{p}_d/\ultrafilter\raa$. 
Thus $\Uproj$ is surjective. 

To see that $\Uproj$ is injective, assume that
$\Uproj(\gee{p}/\ultrafilter)= \Uproj (\gee{q}/\ultrafilter)$. From the definition of
$\Uproj$, this means that
$
\laa (\proj_1\circ\gee{p})/\ultrafilter,\ldots,(\proj_d\circ\gee{p})/\ultrafilter
\raa= \laa (\proj_1\circ\gee{q})/\ultrafilter,\ldots,(\proj_d\circ\gee{q})
/\ultrafilter\raa.
$
Therefore, for every $j\leq d$, $\{i :
\proj_j\circ\gee{p}(i)=\proj_j\circ\gee{q}(i)\}\in\ultrafilter$, and so
\[
\{i : \gee{p}(i)=\gee{q}(i)\} = \bigcap_{j=1}^{d}\{i :
\proj_j\circ\gee{p}(i)=\proj_j\circ\gee{q}(i)\}\in\ultrafilter.
\]
It follows that $\gee{p}/\ultrafilter =\gee{q}/\ultrafilter$, showing that $\Uproj$ is
injective, and hence a bijection as claimed.
\smallskip

To prove that $\Uproj$ is an isomorphism, it remains to prove that
$\Uproj$ respects the relations.
Let us note that $\langD$ 
is the common language of $\geometry_\RelSet(\Model)^{\,\ultrafilter}$ 
and $\geometry_\RelSet\left(\Model^{\,\ultrafilter}\right)$. 
  Let $R=\la\varrho,n\ra\in\RelSet$
be a relation symbol in $\langD$. By our definitions (Sect.~\ref{relations_on_dtuples})
the interpretations of $R$ in $\geometry_{\RelSet}(\Model)$
and $\geometry_\RelSet\left(\Model^{\,\ultrafilter}
\right)$ are $\RMd$ and  $\RMUd$, respectively. 
Let us denote the
interpretation of $R$ in $\geometry_\RelSet(\Model)^{\,\ultrafilter}$
by $\RMdU$, and suppose
$\gee{p_1}\colon I\rightarrow M^d$,
\ldots, $\gee{p_n}\colon I\rightarrow M^d$. 
We have to prove that
\begin{equation}
\label{ee1}
\RMdU(\gee{p_1}/\ultrafilter,\ldots,\gee{p_n}/\ultrafilter)
\ \Longleftrightarrow\  \RMUd
\left(\Uproj(\gee{p_1}/\ultrafilter),\ldots, \Uproj(\gee{p_n}/\ultrafilter)\right).
\end{equation}

This follows directly from \eqref{eq:def-Rf},
\eqref{ultradef} and \eqref{Los}, since we have
\begin{multline*}
\RMdU(\gee{p_1}/\ultrafilter,\ldots, \gee{p_n}/\ultrafilter)\
\Longleftrightarrow\ 
\{i: \RMd(\gee{p_1}(i),\ldots,\gee{p_n}(i))\}\in\ultrafilter\
\Longleftrightarrow\\
\{i :
\Model\models\varrho[\proj_1\circ\gee{p_1}(i),\ldots,\proj_d\circ\gee{p_1}(i),\ldots,
\proj_1\circ\gee{p_n}(i),\ldots,\proj_d\circ\gee{p_n}(i)]\}\in\ultrafilter\
\Longleftrightarrow\\
\Model^{\,\ultrafilter}\models
\varrho\left[(\proj_1\circ\gee{p_1})/\ultrafilter,\ldots,
(\proj_d\circ\gee{p_1})/\ultrafilter,\ldots
      (\proj_1\circ\gee{p_n})/\ultrafilter,\ldots,
(\proj_d\circ\gee{p_n})/\ultrafilter\right]\ \Longleftrightarrow\\
\RMUd
(\laa (\proj_1\circ\gee{p_1})/\ultrafilter,\ldots,
(\proj_d\circ\gee{p_1})/\ultrafilter\raa,\ldots,
      \laa(\proj_1\circ\gee{p_n})/\ultrafilter,\ldots,
(\proj_d\circ\gee{p_n})/\ultrafilter\raa)\
\Longleftrightarrow\\
\RMUd
\left(\Uproj(\gee{p_1}/\ultrafilter),\ldots, \Uproj(\gee{p_n}/\ultrafilter)\right).      
\end{multline*}

Thus \eqref{ee1} holds and $\Uproj\colon
\geometry(\Model)^{\,\ultrafilter}\rightarrow\geometry(\Model^{\,
  \ultrafilter})$ is an isomorphism.
\end{proof}

\medskip
\subsection{Proof of Theorem~\ref{THM1}}
We are now in a position to complete the proof of Theorem~\ref{THM1}.

\begin{blockindent}
\THMone*
\end{blockindent}

\begin{proof}
\label{THM1-proof}
(i) $\Rightarrow$ (ii): Since $\rel$ is definable in $\geometry$ it
must be closed under automorphisms of $\geometry$. So to prove (ii),
it is enough to show that $\rel$ is definable over $\field$. To do so,
we define a translation $\Tr$ from the language of $\geometry$ to the
language of $\field$.  First, for each $i\in\natnumb$, we associate
the list $v_{1+(i-1)d},\dots,v_{id}$ of $d$-many field-variables to
geometry-variable $\gv{v}_i$.  For every relation symbol $S$ in the
language $\geometry$, let $\sigma_S$ be a formula of the language of
$\field$ such that $\sigma_S$ defines $\widehat{S}$. Such a formula
exists by definition because $\geometry$ is a field-definable
coordinate geometry.

Then let $\Tr$ be defined recursively as follows: for every relation symbol $S$ of $\geometry$,
\begin{align*}
   \label{eq:TrR}  \Tr\big(S(\gv{v}_i,\dots)\big) &\quad = \quad \sigma_S(v_{1+(i-1)d},\dots,v_{id},\dots),\\
  \Tr(\gv{v}_i=\gv{v}_j) &\quad =\quad v_{1+(i-1)d}=v_{1+(j-1)d}\land\dots\land v_{id}=v_{jd},\\
    \Tr(\lnot \varphi) & \quad = \quad \lnot\Tr(\varphi)\\
  \Tr(\varphi \land \psi)&\quad = \quad \Tr(\varphi)\land \Tr(\psi), \text{ and}\\
  \exists \gv{v}_i \varphi &\quad =\quad \exists v_{1+(i-1)d}\dots\exists v_{id} \Tr(\varphi).
\end{align*}

Then, by a straightforward formula induction, one can prove
  that, for every formula $\varphi(\gv{v}_1,\ldots,\gv{v}_n)$ in the
  language of $\geometry$ and points $\vec{p}_1, \dots, \vec{p}_n$, we
  have
\begin{equation}\label{eq:tr}
  \geometry\models \varphi [\vec{p}_1,\dots,\vec{p}_n] \iff \field \models \Tr(\varphi)[\hat{p}_1,\dots,\hat{p}_n].  
\end{equation}
For example, if $\varphi$ is atomic formula $S(\gv{v}_1,\ldots,\gv{v}_n)$, then  $\geometry\models S[\vec{p}_1,\dots,\vec{p}_n]$ holds if{}f $(\vec{p}_1,\dots,\vec{p}_n)\in S$ (where $S$ denotes also the interpretation of $S$ in $\geometry$ for notational simplicity). This, by \eqref{eq-flatten}, is equivalent to $(\hat{p}_1,\dots,\hat{p}_n)\in\widehat S$, which holds if{}f $\field\models \sigma_S[\hat{p}_1,\dots,\hat{p}_n]$ since $\sigma_S$ defines $\widehat{S}$ in $\field$. This proves \eqref{eq:tr} for this base case of the formula induction. The other base cases $S(\gv{v}_{i_1},\ldots,\gv{v}_{i_n})$ and $\gv{v}_i=\gv{v}_j$ are similar and the inductions steps are trivial.

Now let $\varrho$ be the formula defining $\rel$ in $\geometry$. Then,
by \eqref{eq:tr}, $\geometry\models
\varrho[\vec{p}_1,\dots,\vec{p}_n]$ if{}f $\field\models
\Tr(\varrho)[\hat{p}_1,\dots,\hat{p}_n]$, \ie $\Tr(\varrho)$ defines
$\widehat \rel$ in $\field$, and hence $\rel$ is definable over
$\field$ as stated.

\smallskip
\noindent
(ii) $\Rightarrow$ (iii): Trivial as affine automorphisms are also automorphisms.

\smallskip
\noindent
(iii) $\Rightarrow$ (i): Assume that $\rel$ is definable over $\field$
and closed under affine automorphisms of $\geometry$.  We have to
prove that $\rel$ is definable in $\geometry$.  As affine
automorphisms of $\geometry$ form a group under composition, they
respect $\rel$.

Let $\lang$ be the language of $\field$, let $R\in\cand$ be the
relation symbol for which $\rel=\RFd$, and let $\RelSet$ be a finite
subset of $\cand$ containing $\keysymbol$, such that $\la\geometry,
\keyrelation \ra\doteq\geometry_{\RelSet}(\field)$. Such $R$ and
$\RelSet$ exist by Remarks~\ref{rem-clearly},
\ref{rem:gammaCol-betaBw} and \ref{prop-doteq}.  Since $\RelSet$
contains $\la\kappa,3\ra$ and $\la\kappa,3\ra_{{\field}{d}}$ coincides
with $\keyrelation$, $\geometry_{\RelSet}(\field)$ is a coordinate
geometry as well, and it follows easily that the definable relations,
automorphisms and affine automorphisms of $\geometry$ and
$\geometry_{\RelSet}(\field)$ coincide.  Hence affine automorphisms of
$\geometry_{\RelSet}(\field)$ respect $\rel=\RFd$, and to prove that
$\rel=\RFd$ is definable in $\geometry$, it is enough to prove that
$\RFd$ is definable in $\geometry_{\RelSet}(\field)$.  Clearly, the
expansion of $\geometry_{\RelSet}(\field)$ with $\RFd$ to the language
$\lang_{\RelSet\cup\{R\}}$ is $\geometry_{\RelSet\cup\{R\}}(\field)$
and every affine automorphism of $\geometry_{\RelSet}(\field)$ is an
affine automorphism of $\geometry_{\RelSet\cup\{R\}}(\field)$, so by
Lemma~\ref{affaut-lemma},
$\field\models(\theta_{\RelSet}\rightarrow\theta_{R})$.  We will apply
Theorem~\ref{simon-tetel} to obtain a proof that $\RFd$ is definable
in $\geometry_{\RelSet}(\field)$, namely we will prove that for any
ultrafilter $\ultrafilter$, every automorphism of
$(\geometry_{\RelSet}(\field))^\ultrafilter$ is an automorphism of
$(\geometry_{\RelSet\cup\{R\}}(\field))^\ultrafilter$.  To prove this,
let $\ultrafilter$ be any ultrafilter (on any index set).  Applying
Lemma~\ref{lemma1} (twice), we observe that the function $\Uproj$ is
both an isomorphism from $(\geometry_{\RelSet}(\field))^\ultrafilter$
to $\geometry_{\RelSet}(\field^\ultrafilter)$, and also an isomorphism
from $(\geometry_{\RelSet\cup\{R\}}(\field))^\ultrafilter$ to
$\geometry_{\RelSet\cup\{R\}}(\field^\ultrafilter)$.

\medskip
It is therefore enough to prove that every automorphism of
$\geometry_{\RelSet}(\field^\ultrafilter)$ is an automorphism of
$\geometry_{\RelSet\cup\{R\}}(\field^{\,\ultrafilter})$.  To this end,
we observe that the model
$\geometry_{\RelSet}(\field^{\,\ultrafilter})$ is a coordinate
geometry over $\field^{\,\ultrafilter}$ because $\RelSet$
contains $\la\kappa,3\ra$ and relation $\laa
\keyformula,3\raa_{\udfield d}$ of
$\geometry_{\RelSet}(\field^{\,\ultrafilter})$ coincides with the
relation $\keyrelation$ of the affine geometry over
$\field^{\ultrafilter}$, cf.\ Remark~\ref{rem:gammaCol-betaBw}.
 
Moreover, we know that
$\field^\ultrafilter\models(\theta_{\RelSet}\rightarrow
\theta_{R})$ because $\field \models
(\theta_{\RelSet}\rightarrow \theta_{R}) $. By
Lemma~\ref{affaut-lemma},  this implies that every affine
automorphism of $\geometry_{\RelSet}(\field^\ultrafilter)$ is an affine
automorphism of $\geometry_{\RelSet\cup\{R\}}(\field^\ultrafilter)$ as well.

Assume therefore that $\alpha$ is an arbitrary automorphism of
$\geometry_{\RelSet}(\field^\ultrafilter)$. By
Proposition~\ref{autlemma}, we can write $\alpha =
A\circ\widetilde{\varphi}$ for some affine automorphism $A$ of
$\geometry_{\RelSet}(\field^{\,\ultrafilter})$ and some automorphism
$\varphi$ of $\field^{\,\ultrafilter}$. Since $A$ is also an affine
automorphism of
$\geometry_{\RelSet\cup\{R\}}(\field^{\,\ultrafilter})$, applying
Proposition~\ref{autlemma} once again shows that $\alpha=
A\circ\widetilde{\varphi}$ is an automorphism of
$\geometry_{\RelSet\cup\{R\}}(\field^{\,\ultrafilter})$ as required.
\end{proof}

\begin{remark}
The condition that $\field$ has more than two elements cannot be
omitted from Theorem~\ref{THM1} when $\field$ is an (unordered)
field. To see this, let us first note that every permutation of
$\F^d$ is an automorphism of the $d$-dimensional affine geometry over
the two element field.  This is so because, in this geometry,
$\Col(\vec{p},\vec{q},\vec{r}\,)$ holds if{}f $\vec{q}=\vec{p}$ or
$\vec{q}=\vec{r}$\ or $\vec{p}=\vec{r}$ since the value
$\lambda$ in the definition of $\Col$ can only take the values $0$ and
$1$.  So lines here are exactly the two element subsets of $F^d$, and
every permutation of $F^d$ respects the relation $\Col$.  In
  other words, $\Col$ here is definable from equality, and hence
  every structure over the two element field is a coordinate
  geometry.

Bearing this in mind, let us consider the following 3-dimensional
geometry over the two element field: $\geometry=\la \{0,1\}^3, O, U_x,
U_y, U_z\ra$, where unary relations $O=\{\la0,0,0\ra\}$, $U_x=\{\la
1,0,0\ra\}$, $U_y=\{\la 0,1,0\ra\}$, $U_z=\{\la 0,0,1\ra\}$ color the
origin and the three standard basis vectors in $F^3$,
respectively. Now $\alpha\in\aut{\geometry}$ if{}f $\alpha$ is a
bijection that fixes points $\la0,0,0\ra$, $\la 1,0,0\ra$, $\la
0,1,0\ra$, $\la 0,0,1\ra$, while $\AffAut{\geometry}=\{\Id\}$ because
no nontrivial affine transformation fixes the origin and all the
three standard basis vectors. Let $\rel=\{\la 1,1,1\ra\}$. Then $\rel$
is not definable in $\geometry$ because it is not respected by the
automorphisms of $\geometry$, yet it is field-definable and respected
by all the affine automorphisms of $\geometry$. So (iii) holds for
$\rel$ but (i) does not.

Nevertheless, the equivalence between (i) and (ii) holds even
  over the two element field because, if $\geometry$ is finite, then
  the definable relations of $\geometry$ are exactly those which are
  closed under the automorphisms of $\geometry$, see
  Remark~\ref{rem:fin}. \qed
\end{remark}

\section{Concluding remarks}

We have introduced FFD coordinate geometries over
fields and ordered fields, and have shown that the definable
relations (concepts) of an FFD coordinate geometry are
exactly those relations that are closed under the automorphisms of the
geometry and are definable over the (ordered) field; moreover
they are exactly the relations closed under the affine automorphism
and are definable over the (ordered) field. These results of
Section~\ref{sec:geometry-automorphisms}, and in particular
Corollary~\ref{cor-erlangen}, yield a simple but powerful strategy for
comparing the conceptual content of different geometries: \emph{If we
wish to compare the sets of concepts of FFD coordinate
geometries over the same field and of the same dimension, it is enough
to compare their affine automorphism groups.}

This insight is strongly connected to the famous Erlangen program of
Felix Klein's because it directly states that to understand the
concepts of these geometries it is enough to understand their affine
automorphisms.

In the second part of this paper, we will apply these results to
investigate how some historically significant geometries including
those of Galilean, Newtonian, Relativistic and Minkowski spacetimes,
as well as Euclidean Geometry are connected conceptually. There, we
are going to introduce a geometry $\Lclst$ corresponding to late
classical kinematics, that contains all the concepts of the above
mentioned ones. So among others, we show the connection depicted by
Figure~\ref{fig:hasse0}.

\begin{figure}
  \begin{tikzpicture}[scale=1]
    \tikzstyle{edge}=[thick]
    \node (LC) at (0,6){$\ca{\Lclst}$};
    \node (R) at (5,3){$\ca{\relst}=\ca{\minkst}$};
    \node (E) at (1,3){$\ca{\euclg}$};
    \node (G) at (-2,2){$\ca{\galst}$};
    \node (N) at (-2,4){$\ca{\newtst}$};
    \node (A) at (0,0) {$\ca{\OAff}$};

    \draw[edge] (A) to (G) to (N) to (LC);
    \draw[edge] (A) to (E) to (LC);
    \draw[edge] (A) to (R) to (LC);
  \end{tikzpicture}
  \caption{Hasse diagram showing how the concept-sets associated with
  various historically significant geometries are related to one another by subset inclusion.}
  \label{fig:hasse0}
\end{figure}
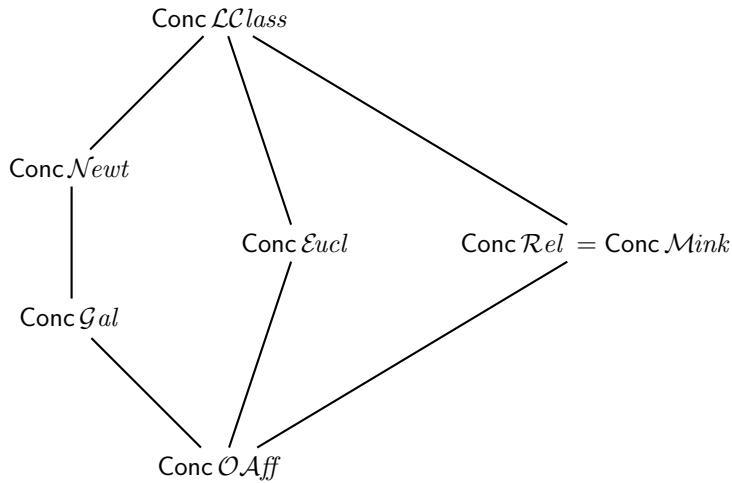
  
There are some natural problems in this research area which are left
open. Here, we list some of them:
\begin{problem}
  Do the results of Section~\ref{sec:geometry-automorphisms} hold for
  all field-definable coordinate geometries?
\end{problem}
The concept-sets of field-definable coordinate geometries form a
complete lattice under $\subseteq$.  This is so because complete
semilattices are complete lattices; and it is a complete
join-semilattice where the supremum of any subset is the geometry
obtained by taking all of the associated relations simultaneously.
\begin{problem}\label{problem:lattice}
  Do the concept-sets of FFD coordinate geometries over $\field$ form
  a lattice under $\subseteq$?
\end{problem}
A positive answer to Problem~\ref{problem:lattice} could be obtained
by answering Problem~\ref{problem:meet} positively.
\begin{problem}\label{problem:meet}
  Let $\geometry_1$ and $\geometry_2$ be two FFD coordinate
  geometries over $\field$. Is there an FFD coordinate geometry $\geometry_3$ such
  that $\ca\geometry_1 \cap \ca\geometry_2 = \ca\geometry_3$?
\end{problem}

\bibliographystyle{amsalpha}
\bibliography{LogRel12019}

\end{document}